\documentclass[a4paper,11pt,twoside,leqno]{article}
\usepackage{amsmath,amssymb,amsfonts,amsthm,graphicx,fancyhdr}
\usepackage[utf8]{inputenc}
\usepackage[T1]{fontenc}
\usepackage{rotating}
\usepackage[colorlinks=true]{hyperref}
\usepackage{mathtools}
\usepackage{thmtools}

\newtheorem{teo}{Theorem}[section]
\newtheorem{lem}[teo]{Lemma}
\newtheorem{prop}[teo]{Proposition}
\newtheorem{cor}[teo]{Corollary}

\newtheorem{ques}[teo]{Question}

\newtheorem{teoA}{Theorem}

\declaretheoremstyle[
  spaceabove=\topsep, spacebelow=\topsep,
  headfont=\bf,  
  notefont=\mdseries, notebraces={(}{)},
  bodyfont=\rmfamily, 
  postheadspace=1em,
  qed=$\Diamond$
]{drem}
\declaretheorem[style=drem, name=Remark, numberlike=teo]{rmk}
\declaretheorem[style=drem, name=Example, numberlike=teo]{exa}

\newcommand{\eg}[0]{\emph{e.g.} }
\newcommand{\ie}[0]{\emph{i.e.} }

\newcommand{\srl}[1]{\overline{#1}}
\newcommand{\jo}[1]{\mathcal{#1}}

\DeclareFontFamily{T1}{mafra}{}
\DeclareFontShape{T1}{mafra}{m}{n}{<->s*[0.95]yswab}{} 
\DeclareFontShape{T1}{mafra}{m}{it}{<->s*[1.0]ygoth}{} 
\DeclareTextFontCommand{\textgoth}{\yfrak}

\DeclareSymbolFont{mafrak}{T1}{mafra}{m}{n}
\DeclareSymbolFontAlphabet{\mathfr}{mafrak}

\DeclareSymbolFont{mbbold}{U}{bbold}{m}{n}
\DeclareSymbolFontAlphabet{\mathbbold}{mbbold}

\newcommand{\mr}[1]{\mathrm{#1}}
\newcommand{\mb}[1]{\mathbbold{#1}}

\newcommand{\mf}[1]{\mathfr{#1}}

\newcommand{\pgen}[1]{\langle #1 \rangle}

\newcommand{\imp}[0]{\Rightarrow}

\newcommand{\eps}[0]{\varepsilon}

\newcommand{\rr}[0]{\ensuremath{\mathbb{R}}}
\newcommand{\zz}[0]{\ensuremath{\mathbb{Z}}}

\newcommand{\del}[0]{\ensuremath{\partial}}
\newcommand{\dd}[0]{\ensuremath{\mathrm{d}}}

\newcommand{\un}[0]{\mathbbold{1}}

\newcommand{\maxx}[1]{\textrm{\raisebox{.5ex}{\mbox{$\underset{#1}{\max}$}}} \:}

\newcommand{\vide}[0]{\varnothing}

\newcommand{\somme}[2]{\overset{#2}{\underset{#1}{\sum}}}

\newcommand{\cay}[0]{\mathrm{Cay}\!}

\begin{document}

\renewcommand{\thefootnote}{\fnsymbol{footnote}}

\renewcommand{\thefootnote}{\arabic{footnote}}

\newcommand{\ud}{\tfrac{1}{2}}
\newcommand{\ut}{\tfrac{1}{3}}
\newcommand{\uq}{\tfrac{1}{4}}

\newcommand{\lmd}{\lambda}
\newcommand{\Lmd}{\Lambda}
\newcommand{\Gm}{\Gamma}
\newcommand{\gm}{\gamma}
\newcommand{\GM}{\Gamma}

\newcommand{\bsl}\backslash
\newcommand{\acts}{\curvearrowright}
\newcommand{\donc}{\rightsquigarrow}

\newcommand{\smdd}[4]{\big( \begin{smallmatrix}#1 & #2 \\ #3 & #4\end{smallmatrix} \big)}

\newcommand{\sgn}{\textrm{sgn}\,}

\begin{center}
\Large Entropy, isoperimetry and the total variation profile\\ of random walks
\vspace*{1cm}

\centerline{\large Antoine Gournay\thanks{The author gratefully acknowledges partial support from the ERC AdG grant 101097307.}}
\end{center}

\vspace*{1cm}

\centerline{\textsc{Abstract}}

\begin{center}
\parbox{12cm}{{ \small
The total variation profile $\mr{TV}_n$ of a random walk measures how far its distribution at time $n$ is from being invariant under the generators.
Three lower bounds are given: an elementary bound $\mr{TV}_n \geq 1/\big(44\,L(n)\big)$ in terms of the speed $L(n)$; a bound by the entropy increment $H(P^{n+1}) - H(P^n)$,
which is a reverse Pinsker inequality;
and a bound in terms of the entropy and the isoperimetric profile.
Deciding whether the entropy-increment bound always improves on the isoperimetric one is
related to conjectures on the growth of the entropy and of groups.
The motivation comes from the question of whether the distributions of a random walk can be transported back onto their translates at bounded cost.
}}
\end{center}

\section{Introduction}\label{s-intro}

Let $G$ be an infinite finitely generated group, $S$ a finite symmetric generating set and $P$ a symmetric finitely supported probability measure whose support generates $G$. The \emph{total variation profile} of the random walk with step distribution $P$ is
\[
\mr{TV}_n := \maxx{s \in S} \big\|\delta_s * P^n - P^n\big\|_{\mr{TV}} .
\]
This profile, named by Hutchcroft \& Lopez \cite{HL}, tends to $0$ exactly when the walk is Liouville, at least for lazy walks (see Derriennic \cite{Der}), and has been studied much less than the speed $L(n) := \mb{E}|P^n|$, the entropy $H(P^n)$, the return probability or the isoperimetric profile. This paper gives lower bounds on it.

The motivation comes from transport amenability \cite{moi-trans}, which asks for probability measures tending to an invariant mean whose translates can be transported back onto them, along the Cayley graph, at a bounded cost (see \S{}\ref{ssrwtransp}, and \cite{Go-erratum} for corrections to the examples of \cite{moi-trans}). The most natural candidates are the distributions $P^n$. Routing the displaced mass through the identity shows that they are suitable as soon as $\sum_x |x|\,|\delta_s * P^n(x) - P^n(x)|$ stays bounded (Lemma \ref{lweightedtc}), which one might be tempted to read as a bound $\mr{TV}_n \preceq 1/L(n)$. This raises the question of how fast $\mr{TV}_n$ can decay. The first lower bound is elementary and holds on every infinite group.

\begin{teoA}\label{tA}
Let $G$ be an infinite finitely generated group and $P$ a symmetric finitely supported probability measure whose support generates $G$. Then $\mr{TV}_n \geq 1/\big(44\,L(n)\big)$ for every $n$ with $L(n) \geq 1$.
\end{teoA}

This is Proposition \ref{ptvspeed}. No assumption is made on the growth, the constant does not depend on the generating set, and laziness is not needed. It shows that the criterion $\mr{TV}_n \preceq 1/L(n)$ above sits exactly at the universal lower bound. For comparison, the main theorem of Hutchcroft \& Lopez \cite{HL} produces sets with boundary ratio $O(\mr{TV}_n)$ and diameter $O(D^*_n \ln(1/\mr{TV}_n))$, where $D^*_n = \max_{m \leq n}L(m)$; since a set of diameter $d$ in a Cayley graph has boundary ratio at least of order $1/d$, this yields $\mr{TV}_n \succeq 1/(D^*_n \ln D^*_n)$, and Theorem \ref{tA} removes the logarithm. A second elementary bound combines the return probability with the volume growth (Lemma \ref{lretgrowth} and Corollary \ref{cretgrowth}).

The classical entropy-increment inequality bounds the profile from above, $\mr{TV}_n^2 \preceq H(P^{n+1}) - H(P^n)$ for a lazy walk, and it has a converse, valid for every walk (Proposition \ref{prevpinsker}):
\[
H(P^{n+1}) - H(P^n) \; \leq \; 4\,|\mr{supp}\,P| \; \mr{TV}_n .
\]
Both ends of $H(P^{n+1}) - H(P^n) \preceq \mr{TV}_n \preceq \big(H(P^{n+1}) - H(P^n)\big)^{1/2}$ are attained (see \S{}\ref{sstv}), even on groups of exponential growth: the upper one, up to a logarithm, on $\zz \wr \zz$ (Example \ref{exzz}), and the lower one on $F \wr \zz$ with $F$ finite (Example \ref{excf}). The converse gives $\mr{TV}_n \geq h/(4|\mr{supp}\,P|)$ in terms of the asymptotic entropy $h$, a quantitative form of the easy half of the entropy criterion of Derriennic \cite{Der80} and Kaimanovich \& Vershik \cite{KV}, and $\mr{TV}_n \succeq 1/\ln n$ on $(\zz/2\zz) \wr \zz^2$ (Example \ref{exz2z2}), where the other bounds only give $\sqrt{\ln n/n}$.

The last bound rests on a reformulation of isoperimetry in terms of entropy. Say that a graph has $\mr{IIS}_\nu$ if $|\del F|/|F| \geq K(1+\ln|F|)^{-1/\nu}$ for every finite set of vertices $F$; groups of exponential growth have $\mr{IIS}_1$ (see Coulhon \& Saloff-Coste \cite{CSC}), and the lamplighters $F \wr \zz^d$ have $\mr{IIS}_d$ (see Erschler \cite{Ers-isop}).

\begin{teoA}\label{tB}
Let $X$ be a graph and $\jo{G}^{\scalebox{.4}{$\searrow$}}(v) := \inf\{ |\del F|/|F| \mid 0 < |F| \leq v\}$ its decreasing isoperimetric ratio function. For every finitely supported $f : X \to [0,1]$ with $\|f\|_{\ell^1} = 1$,
\[
\|\nabla f\|_{\ell^1} \; \geq \; \tfrac14\, \jo{G}^{\scalebox{.4}{$\searrow$}}\big(2e^{2H(f)}\big),
\]
and $\|\nabla f\|_{\ell^1} \geq \tfrac12\, \psi\big(H(f)\big)$ for every convex non-increasing $\psi$ with $\psi(u) \leq \jo{G}^{\scalebox{.4}{$\searrow$}}(2e^u)$. In particular, if $X$ has $\mr{IIS}_\nu$ then $\|\nabla f\|_{\ell^1} \geq \frac K2 \big(1 + \ln 2 + H(f)\big)^{-1/\nu}$, and for a symmetric random walk on a Cayley graph with $\mr{IIS}_\nu$ one has $\mr{TV}_n \succeq H(P^n)^{-1/\nu}$.
\end{teoA}

This is Proposition \ref{pentprofile}, Remark \ref{rentprofile} and Corollary \ref{tlbgradc}. Taking $f = \un_F/|F|$ recovers the isoperimetric inequality, so nothing is lost. For $\nu = 1$ it gives $\mr{TV}_n \succeq 1/H(P^n)$, whereas the main theorem of Hutchcroft \& Lopez \cite{HL} yields $1/(H(P^n)\ln H(P^n))$; for $\nu > 1$ it improves on Theorem \ref{tA}; combined with the entropy bounds of Peres \& Zheng \cite{PZ}, it is never worse than the second elementary bound; and under $\mr{IS}_d$ it recovers the sharp bound $n^{-1/2}$ in groups of polynomial growth. The case of $\mr{IIS}_\nu$ (Proposition \ref{tiisent-p}) and its consequences were obtained in an earlier version of \cite{Go-cuts}, which predates \cite{HL}.

Deciding which of these lower bounds is the best one is related to open problems on the growth of groups. For a group with $\mr{IIS}_\nu$, the entropy-increment bound always improves on Theorem \ref{tB}, \ie $H(P^{n+1}) - H(P^n) \succeq H(P^n)^{-1/\nu}$, if and only if $H(P^n) \succeq n^{\nu/(\nu+1)}$, up to the regularity of $n \mapsto H(P^n)$ (Remark \ref{rrevpinskerfolk}). For groups of exponential growth this is the folklore conjecture $H(P^n) \succeq \sqrt n$, and a positive answer for all $\nu \leq 1$ would imply a weak form of Grigorchuk's gap conjecture \cite{Gri-gap}. In the same spirit, a group on which $\mr{TV}_n \succeq n^{-1/2}$ fails, if there is one and if its speed and entropy grow like powers, must have intermediate growth without $\mr{IIS}_1$, a superdiffusive speed and an entropy growing more slowly than $\sqrt n$ (Remark \ref{rfloor}).

\S{}\ref{ssrwtransp} explains the motivation: transport amenability and random walks. \S{}\ref{sstv} is about the total variation profile: the entropy-increment inequalities, Theorem \ref{tA} and the other elementary bounds. \S{}\ref{ssiis} reformulates isoperimetry in terms of entropy, proves Theorem \ref{tB} and compares it with the other bounds. \S{}\ref{ssrenyi} describes the volume growth, the entropy and the return probability as R\'enyi entropies of $P^n$. \S{}\ref{sques} contains examples, and questions on the growth of the entropy, on the floor $n^{-1/2}$ for $\mr{TV}_n$ and on R\'enyi entropies.

{\sloppy{\itshape Acknowledgements:}
Claude Opus 5.5 was used to assist in drafting parts of this manuscript. The ideas of the proofs are the author's (with the exception of Lemma \ref{lretgrowth} and Proposition \ref{prevpinsker}); the content was reviewed and substantially revised by the author, who is responsible for the final text.\par}

\section{Motivation: random walks and transport}\label{ssrwtransp}

\newcommand{\TC}[0]{\mr{TC}}
\newcommand{\Fol}[0]{\textrm{F\o l}}

Throughout, $G$ is a finitely generated group, $S=S^{-1}$ a finite generating set, $|\cdot|$ the word length and $\Gm = \cay(G,S)$ the Cayley graph whose edges are the pairs $(x,xs)$ with $s \in S$, so that its distance is $d(x,y) = |x^{-1}y|$. Left translations, $\lambda_\gm\xi(x) = \xi(\gm^{-1}x)$, act on $\Gm$ by graph automorphisms. For finitely supported probability measures $\xi$ and $\varphi$, the \textbf{transport cost} $\TC(\xi,\varphi)$ is the infimum of $\sum_{x,y} d(x,y)\,m(x,y)$ over the measures $m$ on $G \times G$ with marginals $\xi$ and $\varphi$; this is the Kantorovich--Rubinstein (or Wasserstein-$1$) distance associated with $d$.

Transport amenability was introduced in \cite{moi-trans} to show the vanishing of the reduced $\ell^p$-cohomology in degree one. In the present conventions, $G$ is \textbf{transport amenable} if there are a constant $K$ and finitely supported probability measures $\eta_n$ such that, for all $s \in S$ and all $n$,
\[
\big\|\lambda_s\eta_n - \eta_n\big\|_{\ell^1G} \xrightarrow[\,n\to\infty\,]{} 0
\qquad\text{and}\qquad
\TC\big(\lambda_s\eta_n,\eta_n\big) \leq K .
\]
(In \cite{moi-trans} the Cayley graph is defined by left multiplication; the inversion $x \mapsto x^{-1}$ exchanges the two conventions.) The measures are moved by automorphisms of $\Gm$ and have to be transported back along its edges; if the translations acted on the same side as the edges, the transport cost would always be at most $1$. Groups with a controlled F\o lner sequence are transport amenable \cite[\S{}4]{moi-trans}, whereas many wreath products, such as $\zz \wr \zz$ and $(\zz/2\zz) \wr \zz^2$, are not, contrary to what is claimed in \cite[Examples 5.1 and 5.2]{moi-trans}; see \cite{Go-erratum}.

Besides normalised characteristic functions of F\o lner sets, there is, on an amenable group, one other usual sequence which might tend to an invariant mean: the distribution of a random walk. Whether it can be used in the definition above reduces in part to a question about the decay of $\|\nabla P^n\|_{\ell^1}$, or equivalently of the total variation profile defined below, which is the subject of the rest of this paper. The point of this section is to make that link explicit.

From now on, $P$ is a symmetric finitely supported probability measure whose support generates $G$, and $P^n = P^{*n}$ is the distribution at time $n$. Write $L(n) := \mb{E}|P^n| = \sum_\gm |\gm|\, P^n(\gm)$ for the speed of the walk and $H(P^n)$ for its entropy. Since $P$ is symmetric, so is $P^n$, \ie $P^n(x^{-1}) = P^n(x)$.

Following Hutchcroft \& Lopez \cite{HL}, the \textbf{total variation profile} of the walk is
\[
\mr{TV}_n := \maxx{s \in S} \big\| \delta_s * P^n - P^n \big\|_{\mr{TV}}
= \tfrac12 \maxx{s \in S} \big\| \delta_s * P^n - P^n \big\|_{\ell^1G} .
\]
(In \cite{HL} it is defined for the lazy simple random walk on a graph, as a supremum over pairs of neighbours; on a Cayley graph the two definitions agree.)
It is non-increasing in $n$, since $\delta_s * P^{n+1} - P^{n+1} = (\delta_s * P^n - P^n) * P$ and convolution by a probability measure does not increase the $\ell^1$-norm.
As $P^n$ is symmetric, left and right translates give the same quantity, and each $\delta_s * P^n - P^n$ is a part of the gradient $\nabla P^n$. With the convention that each edge is counted once, \ie $\|\nabla f\|_{\ell^1} = \sum |f(x) - f(y)|$ over the edges $\{x,y\}$, this gives
\[
\mr{TV}_n \; \leq \; \|\nabla P^n\|_{\ell^1} \; \leq \; |S| \, \mr{TV}_n .
\]
Statements about random walks are therefore formulated with $\mr{TV}_n$, and $\|\nabla \cdot\|_{\ell^1}$ is kept for general functions.

The walk is \textbf{lazy} if $P(\mf{o}) \geq \tfrac12$ (here $\mf{o}$ is used to denote the identity element or the root of the graph). This assumption is made only where it is really needed, and it is then stated explicitly. It matters for \emph{upper} bounds on $\mr{TV}_n$: for a periodic walk, such as the simple random walk on $\zz$, the measures $P^n$ and $\delta_s * P^n$ have disjoint supports and $\mr{TV}_n = 1$ for every $n$. Laziness is therefore assumed in Derriennic's theorem below, in the entropy-increment inequality and in the upper bounds of \S{}\ref{sstv}. It is \emph{not} needed for the monotonicity of $n \mapsto \mr{TV}_n$, nor for the lower bounds on $\mr{TV}_n$ of \S{}\ref{sstv} and \S{}\ref{ssiis}, which use the symmetry of $P$ only through even times.

\begin{lem}\label{lweightedtc}
For any finitely supported probability measure $\eta$ and any $s \in S$,
\[
\TC\big(\lambda_s\eta, \eta\big) \; \leq \;
\sum_{x \in G} |x| \, \big| \lambda_s\eta(x) - \eta(x) \big| .
\]
In particular, a group is transport amenable if there is a constant $K$ such that, for all $s \in S$ and all $n$,
\[\tag{$\star$}\label{ecritrelax}
\somme{x \in G}{} |x| \, \big|\delta_s * P^n(x) - P^n(x)\big| \; \leq \; K .
\]
\end{lem}

\begin{proof}
Write $\mu_\pm$ for the positive and negative parts of $\lambda_s\eta - \eta$; they
have the same total mass $M$. Leave the common mass in place and route everything else through the identity: the product
coupling $m(x,y) = \mu_+(x)\mu_-(y)/M$ has the required marginals and, since
$d(x,y) = |x^{-1}y| \leq |x| + |y|$, a cost of at most
$\sum_x |x|\mu_+(x) + \sum_y |y|\mu_-(y)$.
For the case $\eta = P^n$, note that $\lambda_s P^n = \delta_s * P^n$. It remains to see that \eqref{ecritrelax} forces $\mr{TV}_n \to 0$, so that $P^n$ does tend to an invariant mean: restricting the sum to $|x| \geq r$ gives $K \geq r\big(2\mr{TV}_n - 2P^n(B_r)\big)$, and $P^n(B_r) \to 0$ for every fixed $r$ as $G$ is infinite.
\end{proof}

One could bound $|x|$ by the radius of the support, but that would be very lossy: an element has to be moved by at most its own norm, and for $\eta = P^n$ the norm is of order $L(n)$ on the bulk of the mass, not of order $n$.
In fact, if $L(n)\asymp n$, then the walk is not Liouville (Kaimanovich \& Vershik \cite{KV}, Varopoulos \cite{Var}).
This is relevant since by a theorem of Derriennic \cite{Der} (see also Kaimanovich \& Vershik \cite[\S{}4]{KV}), an aperiodic random walk (for instance a lazy one) is Liouville if and only if $\|\delta_\gm * P^n - P^n\|_{\ell^1G} \to 0$ for every $\gm \in G$, \ie if and only if $\mr{TV}_n \to 0$.

However, what would be relevant for transport are the moments $\sum_x |x|^p\,|\delta_s * P^n(x) - P^n(x)|$ of the gradient, and no bounds on these seem to be available in general.
\begin{rmk}\label{rholder}
H\"older's inequality, with $\tfrac1p + \tfrac1q = 1$
and $M_p(n) := \somme{x}{}|x|^p|\delta_s * P^n(x) - P^n(x)|$, can be used to get
\[
\TC\big(\lambda_s P^n, P^n\big) \; \leq \; \somme{x}{} |x|\,
\big|\delta_s * P^n(x) - P^n(x)\big| \; \leq \; M_p(n)^{1/p} \,
\big(2\mr{TV}_n\big)^{1/q} .
\]
But this does not help if $M_p$ is estimated simply by $M_p(n) \preceq \mb{E}|P^n|^p$. Even in Abelian groups, where it can be checked that the random walk distributions yield a sequence of measures for which the group is transport amenable, the estimate $M_p(n) \preceq \mb{E}|P^n|^p$ is too crude.
\end{rmk}
On the other hand, condition \eqref{ecritrelax} can fail even when the walk is Liouville. By \cite{Go-erratum}, a wreath product $H_0 \wr H$ with $H_0 \neq \{1\}$ and $H$ infinite is transport amenable only if $H_0$ is finite and $H$ is virtually cyclic. So \eqref{ecritrelax} fails, for every choice of $P$, on $\zz \wr \zz$ and on $(\zz/2\zz) \wr \zz^2$, although their random walks are Liouville. On $\zz \wr \zz$, for instance, $\mr{TV}_n \succeq n^{-1/4}$ is much larger than $1/L(n) \asymp n^{-3/4}$ (Example \ref{exzz}).

The next sections discuss bounds on $\mr{TV}_n$.

\newcommand{\IS}[0]{\mathrm{IS}}
\newcommand{\IIS}[0]{\mathrm{IIS}}

\section{Total variation}\label{sstv}

We first gather some known bounds on the total variation profile.

The basic upper bound is the entropy-increment inequality: for a lazy walk,
\[
\mr{TV}_n^2 \; \preceq \; H(P^{n+1}) - H(P^n) .
\]
There are two ways to see it. The first is to apply Pinsker's inequality to the identity $H(P^{n+1}) - H(P^n) = I(X_1 ; X_{n+1})$, where $X_k$ is the position of the walk at time $k$ and $I$ is the mutual information (see Step~1 of the proof of Proposition \ref{prevpinsker} below); this has been found several times, see Erschler \& Karlsson \cite[\S{}2]{EK} and Benjamini, Duminil-Copin, Kozma \& Yadin \cite{BDKY}, and it is a step in Ozawa's functional-analytic proof of Gromov's theorem \cite{Ozawa}.
The second is the discrete analogue of the fact that the time derivative of the entropy of a diffusion is its Fisher information: $H(P^{n+1}) - H(P^n)$ is comparable to $\pgen{\nabla P^n \mid \nabla \ln P^n}$ (this is made rigorous in \cite{EK}), which dominates $4\|\nabla (P^n)^{1/2}\|_{\ell^2}^2$ by the elementary inequality $(b-a)(\ln b - \ln a) \geq 4(\sqrt b - \sqrt a)^2$, and hence a constant multiple of $\|\nabla P^n\|_{\ell^1}^2$ by the Cauchy--Schwarz inequality.
Laziness is genuinely needed here: for the simple random walk on $\zz$ the right-hand side tends to $0$ while $\mr{TV}_n = 1$.

Since $X_1 \to X_{n+1} \to X_{n+2}$ is a Markov chain, the data-processing inequality (see Cover \& Thomas \cite[\S{}2.8]{CoverThomas}) shows that the increments $H(P^{n+1}) - H(P^n)$ are non-increasing in $n$. In particular they are at most $H(P^n)/n$, and $H(P^n) \leq (\ln |S| + 1)L(n) + 1$, so that
\[
\mr{TV}_n \; \leq \; C\sqrt{L(n)/n} \, ,
\]
which is a quantitative form of ``sublinear speed $\imp$ Liouville''.
Note that the step through $H(P^n)/n$ is lossy: on $\zz^d$ the increment version gives the correct $n^{-1/2}$, whereas $\sqrt{L(n)/n}$ gives only $n^{-1/4}$.

Beyond that, sharp upper bounds come from explicit couplings.
For groups of polynomial growth and centred measures one gets $\mr{TV}_n \asymp n^{-1/2}$, together with pointwise Gaussian bounds on the discrete gradient; see Alexopoulos \cite{Alex}, Hebisch \& Saloff-Coste \cite{HSC} and Dungey \cite{Dungey}.

The entropy-increment inequality has a converse. It holds for every walk (neither symmetry nor laziness is needed) and it is sharp in a different regime.

\begin{prop}\label{prevpinsker}
Let $P$ be a finitely supported probability measure on $G$ with $\mr{supp}\,P \subseteq S \cup \{\mf{o}\}$. Then, for every $n \geq 0$,
\[
H(P^{n+1}) - H(P^n) \; \leq \; 4\,|\mr{supp}\,P|\; \mr{TV}_n .
\]
\end{prop}

We could not find Proposition \ref{prevpinsker} stated in the literature.

\begin{proof}
Let $Y_1, Y_2, \ldots$ be independent random elements of $G$ with law $P$, and let $X_k := Y_1 \cdots Y_k$ (with $X_0 := \mf{o}$), so that $X_k$ has law $P^k$. Throughout, write
\[
\mu_s := \delta_s * P^n \quad (s \in \mr{supp}\,P) \qquad \text{and} \qquad \nu := P^{n+1},
\]
and recall that $\|\mu - \mu'\|_{\mr{TV}} = \tfrac12 \|\mu - \mu'\|_{\ell^1G}$.

\emph{Step 1: the increment is a mutual information.} We use the standard notions of Cover \& Thomas \cite[Ch.~2]{CoverThomas}. For random variables $U$ and $V$ with values in countable sets, with joint law $p(u,v)$, marginal laws $p(u)$ and $p(v)$, and conditional laws $p(v \mid u) = p(u,v)/p(u)$, the conditional entropy is $H(V \mid U) := \sum_u p(u)\, H\big(p(\cdot \mid u)\big)$ and the {\bfseries mutual information} is $I(U;V) := H(V) - H(V \mid U)$. Writing $H(V) = -\sum_{u,v} p(u,v) \ln p(v)$, one gets
\[
I(U;V) = \sum_{u,v} p(u,v) \ln \frac{p(v \mid u)}{p(v)} = \sum_u p(u)\, D\big(p(\cdot \mid u) \,\|\, p(\cdot)\big),
\]
where $D(\mu \,\|\, \nu) := \sum_v \mu(v) \ln \frac{\mu(v)}{\nu(v)}$ is the {\bfseries Kullback--Leibler divergence}: the mutual information is the average divergence of the conditional laws of $V$ from its law.

Apply this to $U = X_1$ and $V = X_{n+1}$. One has $X_{n+1} = X_1 Z$ where $Z := Y_2 \cdots Y_{n+1}$ has law $P^n$ and is independent of $X_1 = Y_1$. Hence, conditionally on $X_1 = s$, the law of $X_{n+1}$ is $\mu_s$. Its entropy is $H(P^n)$, because $z \mapsto sz$ is a bijection of $G$ and the entropy of a measure is invariant under bijections. Therefore $H(X_{n+1} \mid X_1) = \sum_s P(s)\, H(\mu_s) = H(P^n)$, and
\[
H(P^{n+1}) - H(P^n) = I(X_1 ; X_{n+1}) = \sum_{s} P(s)\, D(\mu_s \,\|\, \nu).
\]

\emph{Step 2: the likelihood ratios are bounded.} Since $\nu = P * P^n = \sum_t P(t)\, \mu_t$ is a sum of non-negative terms, one has $\nu \geq P(s)\, \mu_s$ pointwise. Hence $\mu_s$ is absolutely continuous with respect to $\nu$, and its density $r_s := \mu_s / \nu$ (defined on the support of $\nu$, which contains that of $\mu_s$) takes values in $[0, 1/P(s)]$. In particular the divergences of Step 1 are finite.

\emph{Step 3: divergence versus $\chi^2$.} Steps 3 and 4 are standard: Step 3 is the classical bound of the Kullback--Leibler divergence by the $\chi^2$-divergence (see \eg Tsybakov \cite[\S{}2.4]{Tsybakov}), and Steps 2--4 together are an instance of the reverse Pinsker inequalities for measures with bounded relative density (see \eg Sason \& Verd\'u \cite{SV}); we include the short arguments for completeness. Let $\mu$ and $\nu$ be probability measures with $\mu = r\nu$. The inequality $\ln u \leq u - 1$, applied at every point where $\mu(x) > 0$, gives
\[
D(\mu \,\|\, \nu) = \sum_x \mu(x) \ln r(x) \; \leq \; \sum_x \mu(x)\big(r(x) - 1\big) = \sum_x \nu(x)\big(r(x)-1\big)^2 + \sum_x \nu(x) \big(r(x) - 1\big),
\]
where the equality uses $\mu = r\nu$ and $r(r-1) = (r-1)^2 + (r-1)$. The last sum is $\mu(G) - \nu(G) = 0$, so that
\[
D(\mu \,\|\, \nu) \; \leq \; \sum_x \nu(x) \big( r(x) - 1\big)^2 .
\]

\emph{Step 4: $\chi^2$ versus total variation.} Bounding one of the two factors $|r - 1|$ by its supremum,
\[
\begin{aligned}
\sum_x \nu(x)\big(r(x) - 1\big)^2 \; & \leq \; \|r - 1\|_{\ell^\infty} \sum_x \nu(x)\, |r(x) - 1| \\
& = \|r - 1\|_{\ell^\infty}\, \|\mu - \nu\|_{\ell^1G} = 2\, \|r-1\|_{\ell^\infty}\, \|\mu - \nu\|_{\mr{TV}} .
\end{aligned}
\]
For $\mu = \mu_s$, Step 2 gives $\|r_s - 1\|_{\ell^\infty} \leq \max\big(1, 1/P(s) - 1\big) \leq 1/P(s)$. Together with Step 3,
\[
D(\mu_s \,\|\, \nu) \; \leq \; \frac{2}{P(s)}\, \|\mu_s - \nu\|_{\mr{TV}} .
\]

\emph{Step 5: comparison with the total variation profile.} Since $\nu = \sum_t P(t)\,\mu_t$ and the total variation distance is convex in each variable,
\[
\|\mu_s - \nu\|_{\mr{TV}} \; \leq \; \sum_t P(t)\, \|\mu_s - \mu_t\|_{\mr{TV}} .
\]
The bijection $x \mapsto t^{-1}x$ preserves the total variation distance and maps $\mu_s$ and $\mu_t$ to $\delta_g * P^n$ and $P^n$, where $g := t^{-1}s$. As $s,t \in S \cup \{\mf{o}\}$, one has $|g| \leq 2$. Writing $g = u_1 \cdots u_k$ with $k = |g|$ and $u_i \in S$, the triangle inequality and the same invariance give
\[
\big\|\delta_g * P^n - P^n\big\|_{\mr{TV}} \; \leq \; \somme{i=1}{k} \big\| \delta_{u_1 \cdots u_i} * P^n - \delta_{u_1 \cdots u_{i-1}} * P^n \big\|_{\mr{TV}} = \somme{i=1}{k} \big\| \delta_{u_i} * P^n - P^n \big\|_{\mr{TV}} \; \leq \; 2\, \mr{TV}_n .
\]
Hence $\|\mu_s - \nu\|_{\mr{TV}} \leq 2\,\mr{TV}_n$ for every $s \in \mr{supp}\,P$.

\emph{Conclusion.} By Steps 1, 4 and 5,
\[
\begin{aligned}[b]
H(P^{n+1}) - H(P^n) & = \sum_{s \in \mr{supp}\,P} P(s)\, D(\mu_s \,\|\, \nu) \\
& \leq \; \sum_{s \in \mr{supp}\,P} P(s) \cdot \frac{2}{P(s)} \cdot 2\, \mr{TV}_n = 4\, |\mr{supp}\,P|\, \mr{TV}_n .
\end{aligned} \qedhere
\]
\end{proof}

Together with the entropy-increment inequality, a lazy walk therefore satisfies
\[
H(P^{n+1}) - H(P^n) \; \preceq \; \mr{TV}_n \; \preceq \; \sqrt{H(P^{n+1}) - H(P^n)} .
\]
The upper end is attained on $\zz^d$, where the increments are of order $1/n$ and $\mr{TV}_n \asymp n^{-1/2}$, and, up to a logarithm, on $\zz \wr \zz$ (Example \ref{exzz}). The lower end is attained for non-Liouville walks (Corollary \ref{crevpinsker} below), and on $(\zz/2\zz) \wr \zz$, where the increments are of order $n^{-1/2}$: the lower end gives the correct order $\mr{TV}_n \asymp n^{-1/2}$, while the upper end only gives $n^{-1/4}$ (see Example \ref{excf}). Roughly speaking, Pinsker's inequality is sharp when $\delta_s * P^n$ differs from $P^n$ by a small displacement of the whole distribution, and Proposition \ref{prevpinsker} is sharp when the difference is concentrated on an event of small probability, such as the state of a lamp.

\begin{cor}\label{crevpinsker}
Let $h := \lim_n \tfrac1n H(P^n)$ be the asymptotic entropy of the walk. Then $\mr{TV}_n \geq h / (4|\mr{supp}\,P|)$ for every $n$.
\end{cor}

\begin{proof}
The increments $H(P^{n+1}) - H(P^n)$ are non-increasing (see above), and a non-increasing sequence whose Ces\`aro means tend to $h$ is bounded below by $h$. Apply Proposition \ref{prevpinsker}.
\end{proof}

This is a quantitative form of the easy half of the entropy criterion of Derriennic \cite{Der80} and Kaimanovich \& Vershik \cite{KV}: a walk with positive asymptotic entropy is not asymptotically invariant.

Summing the entropy-increment inequality and Proposition \ref{prevpinsker} over $k < n$ gives, for a lazy walk,
\[
c \sum_{k<n} \mr{TV}_k^2 \; \leq \; H(P^n) \; \leq \; 4\,|\mr{supp}\,P| \sum_{k<n} \mr{TV}_k ,
\]
so that lower bounds on the total variation profile give lower bounds on the entropy, and upper bounds on the profile, such as those coming from couplings, give upper bounds on the entropy.

Besides Proposition \ref{prevpinsker}, there are also some simple geometric lower bounds.
The first one combines the return probability with the volume growth. Again, we could not find it in the literature, although its main ingredient (bounding off-diagonal transition probabilities by return probabilities, using that $P^N$ is positive semi-definite for $N$ even) is standard (see \eg Hutchcroft \& Lopez \cite[\S{}1]{HL}).

\begin{lem}\label{lretgrowth}
Let $m \geq 1$, let $N \geq 0$ be even, and let $\Phi(v) := \min\{ r \mid |B_r| > v\}$. Then
\[
\mr{TV}_m \; \geq \; \frac{P^{N+m}(\mf{o})}{4\,P^N(\mf{o})\,
\big(\Phi\big(2/P^{N+m}(\mf{o})\big)+1\big)} \, .
\]
\end{lem}

\begin{proof}
Since $P$ is symmetric, $P^{N+m}(\gm) = \pgen{ P^N \mid \lambda_\gm P^m}$ for every $\gm$, so
\[
P^{N+m}(\mf{o}) - P^{N+m}(\gm) = \pgen{P^N \mid P^m - \lambda_\gm P^m}
\leq \|P^N\|_{\ell^\infty} \big\|P^m - \lambda_\gm P^m\big\|_{\ell^1}
\leq 2|\gm| \, P^N(\mf{o}) \, \mr{TV}_m ,
\]
using $\|P^N\|_{\ell^\infty} = P^N(\mf{o})$, valid for $N$ even because $P^N(x) = \pgen{P^{N/2}\delta_x \mid P^{N/2}\delta_{\mf{o}}} \leq P^N(\mf{o})$ by Cauchy--Schwarz, and the telescoping bound $\|P^m - \lambda_\gm P^m\|_{\ell^1} \leq |\gm| \maxx{s \in S}\|P^m - \lambda_s P^m\|_{\ell^1}$.
Set $k := \lfloor P^{N+m}(\mf{o}) / \big(4P^N(\mf{o})\mr{TV}_m\big)\rfloor$.
For $|\gm| \leq k$ the right-hand side is at most $\tfrac12 P^{N+m}(\mf{o})$, so $P^{N+m}(\gm) \geq \tfrac12 P^{N+m}(\mf{o})$, and summing over $B_k$ gives $1 \geq \tfrac12 |B_k| P^{N+m}(\mf{o})$, \ie $|B_k| \leq 2/P^{N+m}(\mf{o})$.
Hence $k \leq \Phi(2/P^{N+m}(\mf{o}))$, which is the claim.
\end{proof}

The natural choice is $N=m$ (for $m$ even). The prefactor $P^{2m}(\mf{o})/P^m(\mf{o})$ is then bounded below whenever the return probability decays polynomially, and the lemma reads $\mr{TV}_m \succeq 1/\Phi\big(2/P^{2m}(\mf{o})\big)$ with $\Phi$ the inverse of the growth. On polynomial growth of degree $D$ one has $P^{2m}(\mf{o}) \asymp m^{-D/2}$ and $\Phi(m^{D/2}) \asymp \sqrt m$, which returns the bound $m^{-1/2}$ (see also Corollary \ref{tlbgradpol-c}); it is sharp, since it matches the upper bound given by the entropy increment.

When the return probability decays faster than any polynomial, as it does on every group of superpolynomial growth, the choice $N = m$ is useless: the prefactor $P^{2m}(\mf{o})/P^m(\mf{o})$ tends to $0$.
Taking $N$ much larger than $m$ brings the prefactor back to a constant, because the return probability is log-convex.

\begin{cor}\label{cretgrowth}
Write $u(N) := -\ln P^N(\mf{o})$. For every even $m \geq 2$ and every even $N$ with $N \geq m\,u(N)$,
\[
\mr{TV}_m \; \geq \; \frac{1}{4e\,\big(\Phi\big(2e^{1+u(N)}\big)+1\big)} \, .
\]
Since $\mr{TV}_n$ is non-increasing, the same holds for odd $m$ provided $N \geq (m+1)\,u(N)$.
\end{cor}

\begin{proof}
Convolution by $P$ is a self-adjoint contraction of $\ell^2G$. If $\mu$ is the spectral measure of $\delta_{\mf{o}}$, then for even $t$ one has $P^t(\mf{o}) = \int_{[-1,1]} (\lambda^2)^{t/2} \dd\mu(\lambda)$, and H\"older's inequality shows that $\psi(t) := \ln P^t(\mf{o})$ is convex on the even integers. (If $P$ is lazy, its spectrum lies in $[0,1]$ and $\psi$ is convex on all integers, but this is not needed.)
Since $\psi(0) = 0$, convexity gives $\psi(N+m) - \psi(N) \geq \tfrac{m}{N}\psi(N) = -m\,u(N)/N \geq -1$. Hence $P^{N+m}(\mf{o})/P^N(\mf{o}) \geq e^{-1}$ and $1/P^{N+m}(\mf{o}) \leq e^{1+u(N)}$, and the claim follows from Lemma \ref{lretgrowth} and the monotonicity of $\Phi$. For odd $m$, apply this to $m+1$.
\end{proof}

The interest of Corollary \ref{cretgrowth} lies in the nature of its proof: on groups whose growth is at least $\mr{exp}(cn^\nu)$, the isoperimetric bound of \S{}\ref{ssiis} is never worse (see Remark \ref{rcretvsiis}).

The second lower bound uses the speed rather than the growth and return probability.

\begin{prop}\label{ptvspeed}
For every $n$ with $L(n) \geq 1$ one has $\mr{TV}_n \geq 1/\big(44\,L(n)\big)$.
\end{prop}

The constant depends on nothing, not even on $S$, and no hypothesis on the
growth is needed. The mechanism is the obvious one: a walk whose distribution is
almost invariant under translation by every element of a ball of radius $k$
cannot have most of its mass in a ball of radius much smaller than $k$.

\begin{proof}
Put $\eps := 2\mr{TV}_n = \maxx{s\in S}\|\lambda_sP^n - P^n\|_{\ell^1}$. Writing
$\gm = s_1 \cdots s_m$ as a geodesic word and telescoping, the left-invariance of
$\|\cdot\|_{\ell^1}$ gives
\[
\big\| \lambda_\gm P^n - P^n \big\|_{\ell^1}
\leq \somme{i=1}{m} \big\| \lambda_{s_i}P^n - P^n\big\|_{\ell^1}
\leq |\gm| \, \eps .
\]
Let $R := \lceil 4L(n)\rceil$. Markov's inequality gives $P^n(B_R) \geq 3/4$.
Since $G$ is infinite there is $\gm$ with $|\gm| = 2R+1$, and then
$\gm B_R \cap B_R = \vide$, as $\gm B_R$ meets $B_R$ only if
$\gm \in B_RB_R^{-1} = B_{2R}$. Now $\lambda_\gm P^n(\gm B_R) = P^n(B_R) \geq
3/4$, so
\[
P^n(\gm B_R) \; \geq \; \tfrac34 - \big\|\lambda_\gm P^n - P^n\big\|_{\ell^1}
\; \geq \; \tfrac34 - (2R+1)\eps .
\]
The two sets being disjoint, $P^n(\gm B_R) + P^n(B_R) \leq 1$, whence
$(2R+1)\eps \geq \tfrac12$. As $L(n)\geq 1$ one has $2R+1 \leq 8L(n)+3 \leq
11L(n)$, so $\eps \geq 1/\big(22L(n)\big)$, \ie $\mr{TV}_n \geq 1/\big(44L(n)\big)$.
\end{proof}

Together with the upper bound above, the speed alone therefore confines the
profile of a lazy walk to
\[
\frac{c}{L(n)} \; \leq \; \mr{TV}_n \; \leq \; C \sqrt{\frac{L(n)}{n}} \, ,
\]
the left-hand end being sharp in this generality.
The two ends are compatible only if $L(n) \succeq n^{1/3}$, which is consistent with, and much weaker than, the universal diffusive lower bound $L(n) \succeq \sqrt n$ of Lee \& Peres \cite{LP}.

\section{Isoperimetry in terms of entropy}\label{ssiis}

\newcommand{\Gc}[0]{\widetilde{\jo{G}}}
\newcommand{\Go}[0]{\jo{G}_0}
\newcommand{\Gd}[0]{\jo{G}^{\scalebox{.4}{$\searrow$}}}

We first recall the isoperimetric notions used in the sequel.

The \textbf{isoperimetric profile} is the function $\jo{F}: \zz_{>0} \to \zz_{\geq 0}$ defined to be the largest function so that $\jo{F}(|F|) \leq |\del F|$, \ie $\jo{F}(x) = \inf \{ |\del F| \, \mid \, |F|= x\}$.
The \textbf{isoperimetric ratio function} is the function $\jo{G}: \zz_{> 0} \to \rr_{\geq 0}$ defined by $\jo{G}(x) := x^{-1}\jo{F}(x) = \inf \{ \frac{|\del F|}{|F|} \, \mid \, |F|= x\}$.
The (decreasing) function $\Gd(x) := \inf \{ \frac{|\del F|}{|F|} \, \mid \, |F|\leq  x\}$ derived from $\jo{G}$ will also come in handy.
Be aware that some texts use ``isoperimetric profile'' for what is called here the isoperimetric ratio function, and that some authors consider rather the (increasing) function $\jo{F}^{\scalebox{.4}{$\nearrow$}}(x) = \inf \{ |\del F| \, \mid \, |F| \geq x\}$, the (increasing) function $J^{\scalebox{.4}{$\nearrow$}}(x) = \Gd(x)^{-1}$ or $J = \jo{G}^{-1}$.

Looking at the decreasing function $\Gd$ instead of $\jo{G}$ does not change much.
Recall that $\jo{F}$ is subadditive: $\jo{F}(a+b) \leq \jo{F}(a) + \jo{F}(b)$. Hence, up to a minor change (taking the concave hull of a subadditive function changes it by at most a factor $2$), $\jo{F}$ is concave.
If $f$ is concave, then $x^{-1} f(x)$ is decreasing.

A graph has a ... \\
\begin{tabular}{@{\textbullet\;}llll}
strong isoperimetric ratio& ($\IS_\omega$) & if $\exists K >0$ so that & $\jo{G}(x) \geq K $. \\
$d$-dimensional isoperimetric ratio&($\IS_d$) &if $\exists K >0$ so that &$\jo{G}(x) \geq \frac{K}{x^{1/d}}$. \\
$\nu$-intermediate isoperimetric ratio& ($\IIS_\nu$) &if $\exists K >0$ so that & $\jo{G}(x) \geq \frac{K}{(1+\ln x)^{1/\nu}}$.
\end{tabular}

$\IS_\omega$ is equivalent to non-amenability.
Typical graphs with $\IS_d$ are Cayley graphs of $\zz^d$.
A Cayley graph of a group which is not virtually nilpotent has $\IS_d$ for every $d$, whereas a virtually nilpotent group has $\IS_d$ exactly for $d$ at most its degree of polynomial growth (this degree is related to the growth of balls, not to the dimension of the continuous counterpart of the group).
By Coulhon \& Saloff-Coste \cite[Th\'eor\`eme 1]{CSC}, a group whose balls satisfy $|B_n| \geq K_1\mr{exp}(K_2 n^\nu)$, with $\nu \in \, ]0,1]$, has $\IIS_\nu$; in particular, groups of exponential growth have $\IIS_1$. The exponent $\nu$ need not be at most $1$, as the isoperimetry of a group can be much better than what its growth predicts: for a non-trivial finite group $F$, the lamplighter $F \wr \zz^d$ has exponential growth and satisfies $\IIS_d$, and no better. Indeed its F\o lner function is $\Fol(k) \asymp \mr{exp}(k^d)$ by Erschler \cite[Theorem 1]{Ers-isop}, which amounts to $\jo{G}(x) \asymp (\ln x)^{-1/d}$. So $1/\nu$ is sometimes smaller than one. Conversely, $\IIS_1$ does not imply exponential growth: the first Grigorchuk group has intermediate growth, but its F\o lner function is at least exponential (see Erschler \& Zheng \cite[Example 1.3]{EZ-iso}), so that it has $\IIS_1$.
There is currently no known group with none of the above three properties: all known amenable groups which are not virtually nilpotent satisfy some $\IIS_\nu$. Some groups (like iterated lamplighters) satisfy even stronger isoperimetric inequalities between $\IS_\omega$ and $\IIS_\nu$.

From now on, only infinite graphs are considered, and the focus is on graphs which satisfy $\IIS_\nu$.
The aim of this section is to express this isoperimetric ratio function directly in terms of the entropy.

\begin{prop}\label{tiisent-p}
Assume $X=(V,E)$ has $\IIS_\nu$ with constant $K$.
Then, for any finitely supported
$f:V \to [0,1]$ with
$\|f\|_{\ell^1V}=1$
and $\|f\|_{\ell^\infty} \leq \mr{exp}(-1)$,
\[
 \| \nabla f\|_1 \geq \frac{K\nu}{2^{1/\nu}(\nu+1)} H(f)^{-1/\nu} .
\]
\end{prop}
Note that except for the constant $K$ which became $\frac{\nu}{2^{1/\nu}(\nu+1)} K$, this is equivalent to the original isoperimetric ratio function. Indeed, taking $f = \un_F/|F|$ (and under the mild assumption that $|F| \geq 3$) this reads
\[
\frac{|\del F|}{|F|} \geq K' ( \ln |F|)^{-1/\nu}.
\]
The following proof can be seen as an adaptation of the Sobolev inequality obtained from $\IS_d$ (see \eg Woess \cite[4.3 Proposition]{Woe}).
In Coulhon \& Saloff-Coste \cite{CSC}, it is mentioned that $\IIS_\nu$ leads to some inequality on the Orlicz norm.
\begin{proof}
Let $[f>t] = \{ x \mid f(x)>t\}$ and by abuse of notation its cardinality too (likewise for $\del [f>t]$).
By assumption, $f$ is not a Dirac mass. Note that
\[
\begin{array}{rll}
\| \nabla f \|_{\ell^1V}
&\displaystyle = \sum_x \sum_{y \sim x \mid f(y) > f(x) } f(y)-f(x) \\
&\displaystyle = \sum_x \sum_{y \sim x \mid f(y) > f(x) } \int_0^\infty \un_{[f(x),f(y)[}(t) \dd t \\
&\displaystyle = \int_0^\infty \sum_x \sum_{y \sim x \mid f(y) > t> f(x) } 1 \dd t \\
&\displaystyle = \int_0^\infty \del [f >t] \dd t \\
&\displaystyle \overset{IS}\geq \int_0^\infty [f >t] / ( 1+\ln [f>t] )^{1/\nu} \dd t\\
&\displaystyle \overset{MI}\geq 2^{-1/\nu} \int_0^1 [f >t] / ( -\ln t )^{1/\nu} \dd t ,
\end{array}
\]
where IS is the use of $\IIS_\nu$ and MI is Markov's inequality $[f>t] \leq \|f\|_1/t= 1/t$. As $[f>t] = \vide$ for $t \geq e^{-1}$, only $t < e^{-1}$ matters; then $-\ln t > 1$, so that $1 + \ln [f>t] \leq 1 - \ln t \leq -2\ln t$ and $(1+\ln [f>t])^{-1/\nu} \geq 2^{-1/\nu}(-\ln t)^{-1/\nu}$. To lighten notation, the factor $K$ coming from $\IIS_\nu$ is omitted above, and so is the factor $2^{-1/\nu}K$ from now on; both are restored in the last line.

Since $\|f\|_\infty  \leq e^{-1}$, for $t\leq e^{-1}$, one has $-\ln t \geq 1$ so $\frac{1+1/\nu}{(-\ln t)^{1/\nu}} \geq \frac{1}{(-\ln t)^{1/\nu}} + \frac{1/\nu}{(-\ln t)^{1+1/\nu}}$.
Hence
\[
 \| \nabla f \|_1
 \geq \int_0^1 [f >t] / ( -\ln t )^{1/\nu} \dd t
 \geq \frac{\nu}{\nu+1} \int_0^1 \bigg(\frac{1}{(-\ln t)^{1/\nu}} + \frac{1/\nu}{(-\ln t)^{1+1/\nu}} \bigg) [f >t] \dd t ,
\]
that is,
\[
\| \nabla f \|_1
\geq \frac{\nu}{\nu+1} \int_0^1 \jo{H}(t)' [f >t] \dd t ,
\]
where $\jo{H}(t) = t / (-\ln t)^{1/\nu}$.

Now, if $t_1 < \dots < t_n$ are the values taken by $f$ and $t_0 := 0$, then
\[
\begin{array}{rll}
\sum_x \frac{ f(x)}{(-\ln f(x) )^{1/\nu}}
&\displaystyle = \sum_{i=1}^n \frac{t_i}{(-\ln  t_i)^{1/\nu}} ( [f>t_{i-1}] - [f>t_i]) \\
&\displaystyle = \sum_{i=0}^{n-1} \bigg( \frac{t_{i+1}}{(-\ln  t_{i+1})^{1/\nu}} - \frac{t_i}{(-\ln  t_i)^{1/\nu}} \bigg) [f>t_i]\\
&\displaystyle = \sum_{i=0}^{n-1} \int_{t_i}^{t_{i+1}} \jo{H}(t)' [f>t_i] \dd t\\
&\displaystyle = \int_{0}^{1} \jo{H}(t)' [f>t] \dd t .
\end{array}
\]
So
\[
\int_0^1 \jo{H}(t)' [f >t] \dd t
= \sum_x f(x) \bigg(-\ln f(x) \bigg)^{-1/\nu}
\overset{\text{J}}{\geq} \bigg( \sum_x f(x) (-\ln f(x)) \bigg)^{-1/\nu} = H(f) ^{-1/\nu} ,
\]
where J stands for Jensen's inequality.
Indeed, as $\phi(x) = x^{-1/\nu}$ is convex, $\sum p_x \phi(a_x) \geq \phi( \sum p_x a_x)$.
Here, $p_x= f(x)$ and $a_x = -\ln f(x)$ were used.
Putting these inequalities together gives
\[
 \| \nabla f\|_1 \geq \frac{\nu}{2^{1/\nu}(\nu+1)} K H(f)^{-1/\nu} ,
\]
as desired.
\end{proof}

The proof of Proposition \ref{tiisent-p} can be streamlined to yield the following statement, which holds for every isoperimetric profile.

\begin{prop}\label{pentprofile}
Let $X = (V,E)$ be a graph and let $\Gd$ be its decreasing isoperimetric ratio function. For every finitely supported $f : V \to [0,1]$ with $\|f\|_{\ell^1V} = 1$,
\[
\|\nabla f\|_1 \; \geq \; \tfrac12 \sum_{x} f(x)\, \Gd\big(2/f(x)\big),
\]
the sum running over the support of $f$. Consequently:
\begin{enumerate}
\item[(a)] $\|\nabla f\|_1 \geq \tfrac14\, \Gd\big(2e^{2H(f)}\big)$;
\item[(b)] if $\psi : [0,\infty[ \, \to [0,\infty[$ is convex, non-increasing and satisfies $\psi(u) \leq \Gd(2e^u)$ for all $u \geq 0$, then $\|\nabla f\|_1 \geq \tfrac12\, \psi\big(H(f)\big)$.
\end{enumerate}
\end{prop}

\begin{proof}
For $t > 0$ let $A_t := [f > t]$. As in the proof of Proposition \ref{tiisent-p}, the coarea formula gives $\|\nabla f\|_1 = \int_0^\infty |\del A_t| \dd t$. Since $\|f\|_{\ell^1} = 1$, Markov's inequality gives $|A_t| \leq 1/t$, and as $\Gd$ is non-increasing, $|\del A_t| \geq \Gd(|A_t|)\,|A_t| \geq \Gd(1/t)\,|A_t|$. By Fubini (the sum runs over the support of $f$),
\[
\begin{array}{lll}
\|\nabla f\|_1 & \displaystyle \geq \; \int_0^\infty \Gd(1/t)\,|A_t| \dd t & \displaystyle = \sum_x \int_0^{f(x)} \Gd(1/t) \dd t \\
& \displaystyle \geq \; \sum_x \int_{f(x)/2}^{f(x)} \Gd(1/t) \dd t & \displaystyle \geq \; \tfrac12 \sum_x f(x)\, \Gd\big(2/f(x)\big),
\end{array}
\]
since $1/t \leq 2/f(x)$ on the last interval of integration.
Now let $Z$ be a random vertex with law $f$ and $Y := -\ln f(Z)$, a non-negative random variable with $\mb{E}Y = H(f)$. With $\phi(u) := \Gd(2e^u)$, which is non-increasing, the right-hand side above is $\tfrac12\mb{E}\,\phi(Y)$.
For (a), Markov's inequality gives $\mb{P}\big(Y \leq 2H(f)\big) \geq \tfrac12$, and $\phi(Y) \geq \phi(2H(f))$ on that event, so that $\mb{E}\,\phi(Y) \geq \tfrac12\phi(2H(f))$.
For (b), $\mb{E}\,\phi(Y) \geq \mb{E}\,\psi(Y) \geq \psi(\mb{E}Y)$ by Jensen's inequality.
\end{proof}

\begin{rmk}\label{rentprofile}
Two choices of $\psi$ recover the previous results. Under $\IIS_\nu$ with constant $K$, the function $\psi(u) := K(1 + \ln 2 + u)^{-1/\nu}$ is convex, non-increasing and bounded by $\Gd(2e^u)$, so (b) gives $\|\nabla f\|_1 \geq \tfrac K2 (1 + \ln 2 + H(f))^{-1/\nu}$: this is Proposition \ref{tiisent-p} up to the constant, and without the assumption $\|f\|_{\ell^\infty} \leq e^{-1}$.
Under $\IS_d$, the function $\psi(u) := K2^{-1/d}e^{-u/d}$ gives $\|\nabla f\|_1 \geq \tfrac K2 2^{-1/d} e^{-H(f)/d}$. A group of polynomial growth of degree $d$ has $\IS_d$ and satisfies $H(P^n) \leq \tfrac d2 \ln n + C$: this follows from Gibbs' inequality $H(\mu) \leq -\sum_x \mu(x) \ln \sigma(x)$, valid for every probability measure $\sigma$, applied with $\sigma(x)$ proportional to $e^{-|x|/\sqrt n}$ (whose normalising constant is of order $n^{d/2}$), together with $L(n) \preceq \sqrt n$ (see Hebisch \& Saloff-Coste \cite{HSC}). Hence $\mr{TV}_n \succeq n^{-1/2}$, which is the bound of Corollary \ref{tlbgradpol-c} below, with an upper bound on the entropy replacing the lower bound on the return probability used there.
Finally, on every group, the isoperimetric inequality of Coulhon \& Saloff-Coste \cite[Th\'eor\`eme 1]{CSC} gives $\Gd(v) \geq c/\Phi(2v)$, where $\Phi$ is the inverse of the growth function, and (a) then yields $\mr{TV}_n \succeq 1/\Phi\big(4e^{2H(P^n)}\big)$; compare with Corollary \ref{cretgrowth}.
\end{rmk}

We now turn to explicit lower bounds on $\mr{TV}_n$ (equivalently, on $\| \nabla P^n\|_{\ell^1}$), starting with the $\IS_d$ case.
\begin{cor}\label{tlbgradpol-c}
If $G$ has polynomial growth of degree $d \geq 2$ (so that its Cayley graphs have $\IS_d$), then $\mr{TV}_n \geq K' n^{-1/2}$ for $n$ even.
\end{cor}
\begin{proof}
Under $\IS_d$ with $d \geq 2$, one has $\|\nabla P^n\|_1 \geq \|P^n\|_{d/(d-1)} =: N_d(n)$ (see Woess \cite[4.3 Proposition]{Woe}).
By interpolation,
$\|P^n\|_2 \leq \|P^n\|_{d/(d-1)}^t  \|P^n\|_\infty^{1-t}$
where $\frac{1}{2} = \frac{d-1}{d} t + \frac{1}{\infty} (1-t)$, \ie $t = \frac{d}{2(d-1)}$ and, consequently, $1-t = \frac{d-2}{2(d-1)}$.
Since, for $n$ even, $\|P^n\|_2^2 = P^{2n}(\mf{o}) \asymp n^{-d/2}$ and $\|P^n\|_\infty = P^n(\mf{o}) \asymp n^{-d/2}$ (see \eg \cite{HSC}), this means that
\[
2^{-(d-2)/2}K_1^{d-2-2/d} n^{-1/2} \leq N_d(n) .
\]
Hence, for $n$ even, $\|\nabla P^n\|_1 \geq K_2 n^{-1/2}$, and $\mr{TV}_n \geq K_2 n^{-1/2}/|S|$.
\end{proof}
The restriction to even $n$ is harmless: since $\mr{TV}_n$ is non-increasing (see \S{}\ref{ssrwtransp}), for odd $n$ one has $\mr{TV}_n \geq \mr{TV}_{n+1} \geq K'(n+1)^{-1/2} \geq \tfrac{K'}{\sqrt 2}\, n^{-1/2}$. Laziness is not needed for this.

\begin{cor}\label{tlbgradc}
Let $X$ be a graph with $\IIS_\nu$ (with constant $K$), let $P$ be the transition kernel of a random walk on $X$ and let $o$ be a vertex. Then
\[
\big\|\nabla P^n(o,\cdot)\big\|_{\ell^1} \; \geq \; \frac{K\nu}{2^{1/\nu}(\nu+1)} \, H\big(P^n(o,\cdot)\big)^{-1/\nu}
\]
for every $n$ with $\|P^n(o,\cdot)\|_{\ell^\infty} \leq e^{-1}$. If $X$ is a Cayley graph of a group and $P$ is symmetric, this reads
\[
\mr{TV}_n \; \geq \; \frac{K\nu}{2^{1/\nu}(\nu+1)\,|S|} \, H(P^n)^{-1/\nu},
\]
and the condition on $n$ holds as soon as $n$ is large enough.
\end{cor}
\begin{proof}
The first inequality is Proposition \ref{tiisent-p} applied to $f = P^n(o,\cdot)$. On a Cayley graph with $P$ symmetric, $P^n(o,\cdot)$ is the translate of $P^n$ and $\|\nabla P^n\|_{\ell^1} \leq |S|\,\mr{TV}_n$ (see \S{}\ref{ssrwtransp}); and $\|P^n\|_{\ell^\infty} \to 0$ since the group is infinite.
\end{proof}

Hutchcroft \& Lopez \cite{HL} gave the following bound: for an amenable graph of degree at most $M$ and every $n$ there is a set $W_n$ with $|\del W_n|/|W_n| \leq 4\mr{TV}_n$ and $\ln |W_n| = O\big( (H(P^n) + \ln n)\ln(1/\mr{TV}_n)\big)$. Rearranged, this reads
\[
\Fol\big( 1/(4\mr{TV}_n) \big) \; \leq \; \mr{exp}\Big( C\big(H(P^n) + \ln
n\big)\ln\big(1/\mr{TV}_n\big) \Big),
\]
so that, for a group whose F\o lner function is at \emph{least} exponential, \eg any group of exponential growth, it gives $\mr{TV}_n \succeq 1/(H(P^n)\ln H(P^n))$. Corollary \ref{tlbgradc} (with $\nu = 1$) removes this logarithm.

Corollary \ref{tlbgradc} should also be compared with Proposition \ref{ptvspeed}: as $H(P^n) \leq (\ln|S|+1)L(n)+1$, it implies $\mr{TV}_n \succeq L(n)^{-1/\nu}$ for a group with $\IIS_\nu$, which is the same bound when $\nu = 1$ and a better one when $\nu > 1$; the speed bound is however much more direct, and holds on every group. The case $\nu > 1$ happens, for instance, for the lamplighters $F \wr \zz^d$ with $d \geq 2$ (see above). On these, Proposition \ref{prevpinsker} does even better: see Example \ref{exz2z2} for $d = 2$, while for $d \geq 3$ the walk is not Liouville and $\mr{TV}_n$ is bounded below by a positive constant (Corollary \ref{crevpinsker}). More generally, Corollary \ref{tlbgradc} is far from sharp whenever the asymptotic entropy is positive.

\begin{rmk}\label{rcretvsiis}
For groups whose volume growth is at least $\mr{exp}(cn^\nu)$ with $\nu \leq 1$ (and which therefore satisfy $\IIS_\nu$), Corollary \ref{cretgrowth} is never better than Corollary \ref{tlbgradc}. Indeed, Peres \& Zheng \cite{PZ} show that on groups a return probability $P^{2n}(\mf{o}) \geq \mr{exp}(-Cn^\vartheta)$ forces $H(P^n) \preceq n^{\vartheta/(1-\vartheta)}$ (a bound which is only informative for $\vartheta < 1/2$, see Hutchcroft \& Lopez \cite[\S{}1]{HL}). Hence Corollary \ref{tlbgradc} gives $\mr{TV}_n \succeq H(P^n)^{-1/\nu} \succeq n^{-\vartheta/((1-\vartheta)\nu)}$, which is the bound provided by Corollary \ref{cretgrowth} under the same assumptions; and Corollary \ref{tlbgradc} is stronger as soon as the entropy grows more slowly than the bound of Peres \& Zheng.
\end{rmk}

\begin{rmk}\label{rentlower}
Combining Corollary \ref{tlbgradc} with the entropy-increment inequality of \S{}\ref{sstv} (for a lazy walk) gives $H(P^{n+1}) - H(P^n) \succeq \mr{TV}_n^2 \succeq H(P^n)^{-2/\nu}$, whence $H(P^n) \succeq n^{\nu/(\nu+2)}$. Obtaining lower bounds on the entropy from isoperimetry was one of the motivations for Proposition \ref{tiisent-p} (see Question \ref{quesentgro}). The same bound follows more directly from $H(P^n) \geq -\ln \|P^n\|_{\ell^2}^2 = -\ln P^{2n}(\mf{o})$ and the upper bound $P^{2n}(\mf{o}) \leq \mr{exp}(-cn^{\nu/(\nu+2)})$ implied by $\IIS_\nu$ (see Pittet \& Saloff-Coste \cite{PSC-survey}).
\end{rmk}

In view of Shalom \& Tao \cite[\S{}6]{ST}, it might be useful to have bounds on $\|\nabla g_n\|_1$ where $g_n = \tfrac{1}{n} \sum_{i=0}^{n-1} P^i$ (it is a partial sum of Green's kernel renormalised to be a probability measure).
\begin{cor}\label{tlbgradgreen-c}
Let $g_n$ be as above.
If $G$ has $\IIS_\nu$ (\eg $|B_n| \geq K_1 \mr{exp}(K_2n^\nu)$ with $\nu \leq 1$),
then $\|\nabla g_n\|_1 \geq K H(P^n)^{-1/\nu}$.
\end{cor}
\begin{proof}
The main point is an upper bound on $H(g_n)$ (note that $\|g_n\|_{\ell^1} = 1$, and that $\|g_n\|_{\ell^\infty} \leq e^{-1}$ for $n$ large).
For a convex combination of probability measures,
\[
H\Big(\sum_i w_i\mu_i\Big) = \sum_i w_iH(\mu_i) + \sum_i w_i D\Big(\mu_i \,\Big\|\, \sum_j w_j\mu_j\Big),
\]
where $D$ is the Kullback--Leibler divergence. Here $w_i = \tfrac1n$ and $\mu_i = P^i$. Since $g_n \geq \tfrac1n P^i$, each divergence is at most $\ln n$, and $H(P^i) \leq H(P^n)$ as the entropy of the walk is non-decreasing. Hence $H(g_n) \leq H(P^n) + \ln n \leq 2H(P^n)$ for $n$ large, because $H(P^n) \geq -\ln P^{2n}(\mf{o})$ grows faster than $\ln n$ under $\IIS_\nu$ (see Remark \ref{rentlower}).
Proposition \ref{tiisent-p} then gives the claim.
\end{proof}
Note that the entropy grows more slowly than the logarithm of the volume when the speed of the random walk is not linear (see \cite[\S{}4]{Go-comp} for the explicit estimates).
For example, if a group has exponential growth and speed $n^{\srl{\beta}}$, then
$\|\nabla g_n\|_1 \geq K/n^{\srl{\beta}}$.
This indicates that the ``optimistic'' bound $\|\nabla g_n\|_1 \leq 1/n$ probably seldom holds (so that the different cases in \cite[\S{}6]{ST} are necessary).

This estimate also contrasts with the fact that $\|\Delta g_n\|_1 \leq K/n$ (for some constant $K$ which depends on the way the Laplacian $\Delta$ is defined).

\section{R\'enyi entropy and groups}\label{ssrenyi}

\newcommand{\HR}{H}

The main aim of this section is to show that three invariants of groups (growth, return probability and entropy) are essentially instances of R\'enyi entropies.

For a probability measure $\mu$, the entropy is
\[
\HR(\mu) = \sum_x \mu(x) (-\ln \mu(x)) .
\]
Denote for $q \in \rr \setminus \{0\}$:
\[
 \|\mu\|_q = \bigg( \sum_{x \mid \mu(x) \neq 0}  \mu(x)^q \bigg)^{1/q} .
\]
Also $\|\mu\|_0 = |\mathrm{supp}\, \mu|$.

The restriction $\mu(x) \neq 0$ is only necessary for $q \leq 0$.
Being a probability measure implies $\|\mu\|_1  =1$.
For $p \in \rr \setminus \{0,1\}$, define the R\'enyi entropy $\HR_p$ of $\mu$ to be
\[
\HR_p(\mu) = \frac{-\ln \|\mu\|_p^p}{p-1} .
\]
Another possible way of writing this is $\frac{-\ln \|\mu\|_p^p}{p-1} = -p' \ln \|\mu\|_p$, where $p'$ is the H\"older conjugate of $p$.


The following lemma (whose proof could be left as an exercise) recalls why this has anything to do with entropy.
\begin{lem}
Let $q \in \, ]-\infty,1[ \setminus \{0\}$ and $p \in \, ]1,\infty[$. Then
\[\tag{*}\label{ineqre}
\HR_q(\mu) \geq H(\mu) \geq \HR_p(\mu) .
\]
Also
\[
\begin{array}{ccc}
\displaystyle \lim_{p \to 1^\pm} \HR_p(\mu) = H(\mu),
& &
\displaystyle \lim_{q \to 0^\pm} \HR_q(\mu) = \ln \| \mu \|_0,\\[2ex]
\displaystyle  \lim_{q \to +\infty} \frac{\ln \|\mu\|_q^q}{q-1} = \ln \|\mu \|_\infty
& \text{ and } &
\displaystyle  \lim_{q \to -\infty} \frac{\ln \|\mu\|_q^q}{q-1} = \ln \|\mu \|_{-\infty}
\end{array}
\]
where $\| \mu \|_\infty = \sup_x |\mu(x)|$ and $\|\mu\|_{-\infty} = \inf_x \{ |\mu(x)| \mid \mu(x)\neq 0\}$.

\end{lem}
\begin{proof}
Let $f$ be defined by $f = \mu^q$ for $q \neq 0,1$. Then
\[
H(\mu) = \sum_x \mu(x) (-\ln f(x)^{1/q}) = \frac{1/q}{1-1/q} \sum_x \mu(x) (-\ln f(x)^{1-1/q}) .
\]
If $\frac{1}{q-1} > 0$ (\ie $q>1$), then use convexity of $x \to -\ln x$ to get
\[
\begin{aligned}
H(\mu) & = \frac{-1}{q-1} \sum_x \mu(x) \ln f(x)^{1-1/q}
\geq \frac{-1}{q-1}  \ln \bigg( \sum_x \mu(x) f(x)^{1-1/q} \bigg) \\
& = \frac{-1}{q-1}  \ln \bigg( \sum_x \mu(x)^q \bigg) ,
\end{aligned}
\]
\ie $H(\mu) \geq \HR_q(\mu)$, as desired. If $\frac{1}{q-1} < 0$ (\ie $q < 1$ and $q \neq 0$), then use concavity of $x \to \ln x$ to get
\[
\begin{aligned}
H(\mu) & = \frac{-1}{q-1} \sum_x \mu(x) \ln f(x)^{1-1/q}
\leq \frac{-1}{q-1}  \ln \bigg( \sum_x \mu(x) f(x)^{1-1/q} \bigg) \\
& = \frac{-1}{q-1}  \ln \bigg( \sum_x \mu(x)^q \bigg) ,
\end{aligned}
\]
\ie $H(\mu) \leq \HR_q(\mu)$. For the limit, apply l'H\^opital's rule:
\[
\lim_{p \to 1^\pm} \frac{-\ln \sum_x \mu(x)^p}{p-1}
\overset{H}{=}
\lim_{p \to 1^\pm} - \frac{\sum_x \mu(x)^p \ln \mu(x) }{\|\mu\|^p_p} = \frac{\sum_x \mu(x) (- \ln \mu(x) ) }{\|\mu\|_1} =  H(\mu).
\]
The other limits are also straightforward computations.
\end{proof}
Note that \eqref{ineqre} is stable under using the interpolation of $p$-norm from the $1$-norm and an $r$-norm (with $r>p>1$). This is a convoluted way of showing that the limits as $p\to 1^\pm$ are monotone. More precisely:
\begin{lem}\label{lrenyi}
If $p<q$, then
\[
\HR_q(f) \leq \HR_p(f) .
\]
If further $1<p<q$, then
\[
\HR_p(f) \leq \tfrac{p(q-1)}{(p-1)q} \HR_q(f).
\]
\end{lem}
\begin{proof}
Recall Lyapunov's interpolation inequality: if $q,r \in \, ]-\infty,\infty[$ and $p= \lambda r+ (1-\lambda) q$ (with $\lambda \in \, ]0,1[$), using H\"older with exponent $1/\lambda >1$, one gets
\[
 \int |f|^p =  \int |f|^{\lambda r} |f|^{(1-\lambda)q}  \leq \left( \int |f|^r\right)^\lambda  \left(\int |f|^q\right)^{1-\lambda} = \| f \|_r^{r\lambda} \| f\|_q^{q(1-\lambda)}.
\]
Let us first show that the ratio is monotone on $]1,\infty[$. Assume $p = \theta 1 + (1-\theta)q$, \ie $\theta = \tfrac{q-p}{q-1}$, then
\[
\|f\|_p^p \leq \|f\|_{1}^{1 \cdot\theta}\cdot\|f\|_{q}^{q(1-\theta)} = \|f\|_{q}^{q(1-\theta)} ,
\]
using $\|f\|_1=1$. Since $1-\theta = \tfrac{1-p}{1-q}$, one has $\ln \|f\|_p^p \leq \tfrac{1-p}{1-q} \ln \|f\|^q_q$ which gives (since $p>1$)
\[
\frac{-\ln \|f\|_p^p}{p-1} \geq \frac{ -\ln \|f\|^q_q}{q-1} .
\]
Next, let us show they are equivalent up to a constant on $]1,\infty[$.
For $p<q$, $\|f\|_p \geq \|f\|_q$ so $ \ln \|f\|_p^p \geq \tfrac{p}{q} \ln \|f\|_q^q$ and so, if $p>1$,
\[
\frac{-\ln \|f\|_p^p}{p-1} \leq \tfrac{p(q-1)}{(p-1)q} \frac{-\ln \|f\|_q^q}{q-1} .
\]
It turns out that using this inequality for $p<1$ is not interesting.
A better way to finish on the interval $]-\infty,1[$ is to note that Lyapunov's inequality still holds: $\ln \|f\|_p^p \leq \tfrac{1-p}{1-q} \ln \|f\|^q_q$ if $p>q$. This time $q<p<1$, so
\[
\frac{\ln \|f\|_p^p}{1-p} \leq \frac{\ln \|f\|^q_q}{1-q}. \qedhere
\]
\end{proof}

Let $P$ be as in \S{}\ref{ssrwtransp}, and assume moreover that it is lazy and that its support is $S \cup \{\mf{o}\}$.
From the distributions $P^n$, one commonly studies two functions: the entropy $n \mapsto H(P^n)$ and the return probability $n \mapsto P^n(\mf{o})$.
As $P$ is symmetric and lazy, convolution by $P$ is a positive operator on $\ell^2G$, so that
$P^n(\mf{o}) = \|P^n\|_{\ell^\infty}$; moreover
$P^{2n}(\mf{o}) = \|P^n\|_{\ell^2}^2$.

Another important function in geometric group theory is the growth function $n \mapsto |B_n|$, where $B_n$ is the ball of radius $n$; by the assumption on the support of $P$, the support of $P^n$ is $B_n$.

\begin{rmk}
It is interesting to notice that each of these functions is, up to a logarithm and a sign, a R\'enyi entropy, namely:\\
\textbullet $H_0(P^n)$ is the logarithm of the growth function.\\
\textbullet $H_1(P^n)$ is the entropy.\\
\textbullet $H_2(P^n)$ is minus the logarithm of the return probability at time $2n$.\\
\textbullet $H_\infty(P^n)$ is minus the logarithm of the return probability.\\
This is noteworthy in view of Question \ref{qrent}. The logarithms may look artificial, but the logarithm of the growth function (resp. minus the logarithm of the return probability) are subadditive functions.
\end{rmk}

\section{Examples, remarks and questions}\label{sques}

On the topic of entropy, the author has often heard the following \\
{\bf Conjecture [folklore]:} If $G$ has exponential growth, then $H(P^n) \geq K n^{1/2}$.\\
It is also stated as an open problem by Zheng \cite[\S{}3.2]{Zheng-ICM}. This motivates two possible behaviours for groups of growth $\geq \mathrm{exp}(n^v)$ (for $v \in \, ]0,1]$):
\begin{ques}\label{quesentgro}
Assume $|B_n|  \geq K_1\mr{exp}(K_2n^v)$, does
\begin{enumerate}
 \item $H(P^n) \geq K n^{v/2}$?
 \item $H(P^n) \geq K n^{v/(1+v)}$?
 \end{enumerate}
\end{ques}
Question \ref{quesentgro}.2 is motivated by the fact that $\IIS_\nu$ implies a volume growth of the type $|B_n| \succeq \mathrm{exp}(n^{\nu/(1+\nu)})$: one naturally asks whether this extends to the entropy (which is sometimes described as a form of volume growth).

Note that a positive answer to Question \ref{quesentgro}.2 would imply a weak form of the gap conjecture.
Indeed, by Erschler \cite[Lemma 5.1]{Ers-sp} or \cite[Lemma 4.4]{Go-comp}, $H(P^n) \preceq \ln |B_{L(n)}|$, and $L(n)^2 \preceq n\,H(P^n)$ by the Varopoulos--Carne estimate (see \cite[\S{}1]{HL}); hence $|B_n| \leq K_3 \mr{exp}(K_4 n^V)$ implies $H(P^n) \leq K' n^{V/(2-V)}$.
Combining this with a hypothetical positive answer to Question \ref{quesentgro}.2 for a group with $K_1 \mr{exp}(K_2 n^v) \leq |B_n| \leq K_3 \mr{exp}(K_4 n^V)$, one gets $\frac{V}{2-V} \geq \frac{v}{1+v}$, \ie $V \geq \frac{2v}{1+2v}$.
Under the (very strong) assumption that $V = v$, this would in turn forbid $V = v < 1/2$.
It does not rule out arbitrarily small growth (\eg $v=1/5$ only gives $V \geq 2/7$); it simply forces oscillations in the growth function.

\begin{rmk}\label{rrevpinskerfolk}
For a group with $\IIS_\nu$, Proposition \ref{prevpinsker} and Corollary \ref{tlbgradc} give two lower bounds on the total variation profile, $\mr{TV}_n \succeq H(P^{n+1}) - H(P^n)$ and $\mr{TV}_n \succeq H(P^n)^{-1/\nu}$. Deciding whether the first one always beats the second one, \ie whether
\[\tag{$\dagger$}
H(P^{n+1}) - H(P^n) \; \succeq \; H(P^n)^{-1/\nu},
\]
amounts to answering Question \ref{quesentgro}.2 (with $v = \nu$), and in particular, for groups of exponential growth (which satisfy $\IIS_1$), to deciding the folklore conjecture.
Indeed, put $q := 1 + \tfrac1\nu$. If $H(P^{k+1}) - H(P^k) \geq c\, H(P^k)^{-1/\nu}$ for all $k$, then the convexity of $x \mapsto x^q$ gives
\[
H(P^{k+1})^q - H(P^k)^q \; \geq \; q\, H(P^k)^{q-1}\big(H(P^{k+1}) - H(P^k)\big) \; \geq \; qc ,
\]
and summing over $k$ yields $H(P^n)^q \succeq n$, \ie $H(P^n) \succeq n^{\nu/(\nu+1)}$. Conversely, if $H(P^n) \asymp n^\eta$ with $\eta \geq \nu/(\nu+1)$, then, the increments being non-increasing, $H(P^{n+1}) - H(P^n) \geq \frac{H(P^{Kn}) - H(P^n)}{(K-1)n} \succeq n^{\eta-1} \succeq H(P^n)^{-1/\nu}$ for a fixed $K$ large enough. The two statements are thus equivalent, up to the regularity of $n \mapsto H(P^n)$.
Since a group of growth at least $\mr{exp}(n^\nu)$ satisfies $\IIS_\nu$ for $\nu \leq 1$ (see \S{}\ref{ssiis}), a positive answer to ($\dagger$) would therefore give strong support to Grigorchuk's gap conjecture \cite{Gri-gap}: it would imply its weak form described above.
In other words, by the computation above, a group whose growth is of the form $\mr{exp}(n^v)$ with $v < 1/2$ would be a counterexample to ($\dagger$) with $\nu = v$.
\end{rmk}

Upper bounds on the total variation profile also give lower bounds on the entropy: on a group with $\IIS_\nu$, an estimate $\mr{TV}_n \preceq n^{-a}$ implies $H(P^n) \succeq n^{a\nu}$, by Corollary \ref{tlbgradc}.

We now turn to examples, for which the following observation is useful: much like the speed and the entropy, the total variation profile is monotone under quotients.
\begin{lem}\label{lquot}
Let $\pi : G \to Q$ be onto, $P$ a finitely supported measure on $G$ and $\pi_*P$ the push-forward measure on $Q$. Then $\mr{TV}_n(P) \geq \mr{TV}_n(\pi_*P)$.
\end{lem}
\begin{proof}
The push-forward $\pi_*P$ is a finitely supported probability measure on $Q$, its convolution powers are the push-forwards $\pi_*(P^n)$, and $\pi_*(\delta_s * P^n) = \delta_{\pi(s)} * \pi_*(P^n)$. Since push-forwards do not increase the total variation distance, the total variation profile of $(G,P,S)$ dominates that of $(Q, \pi_*P, \pi(S))$.
\end{proof}

\begin{exa}\label{exz2z2}
Consider the lazy simple random walk on $(\zz/2\zz) \wr \zz^2$ with respect to the standard generators (switch the lamp at the position of the lamplighter, or move the lamplighter), and let $R_n$ be the range of the lamplighter up to time $n$. Then $H(P^n) \asymp \mb{E}|R_n|$. For the lower bound, condition on the trajectory of the lamplighter: the lamp states at distinct visited sites are then independent, and a positive proportion of them are non-degenerate, so conditioning (which does not increase entropy) gives $H(P^n) \geq c\,\mb{E}|R_n|$. For the upper bound, $X_n$ is determined by its position, the range and the lamp states on the range, and the number of connected subsets of $\zz^2$ of size $k$ containing the origin is at most exponential in $k$, so $H(P^n) \leq C\,\mb{E}|R_n| + C\ln n$. By Dvoretzky \& Erd\H{o}s \cite{DE}, $\mb{E}|R_n| \asymp n/\ln n$. These lamplighters go back to Kaimanovich \& Vershik \cite[\S{}6]{KV}, and the asymptotics of their entropy are well known (see Erschler \cite{Ers-drift}, and Benjamini, Duminil-Copin, Kozma \& Yadin \cite[Proposition 13]{BDKY-lamp}, where iterated wreath products are also treated); the argument above is recalled for completeness.

Since the increments are non-increasing, $H(P^{n+1}) - H(P^n) \geq \frac{H(P^{Kn}) - H(P^n)}{(K-1)n}$ for every integer $K \geq 2$, and the two-sided bound $H(P^n) \asymp n/\ln n$ makes the right-hand side $\succeq 1/\ln n$ for a fixed $K$ large enough. Proposition \ref{prevpinsker} therefore gives
\[
\frac{1}{\ln n} \; \preceq \; \mr{TV}_n \; \preceq \; \frac{1}{\sqrt{\ln n}} ,
\]
the upper bound coming from the entropy-increment inequality. For comparison, Corollary \ref{tlbgradc} (which applies with $\nu = 2$, see \S{}\ref{ssiis}) only gives $\mr{TV}_n \succeq \sqrt{\ln n / n}$, and Proposition \ref{ptvspeed} only gives $\mr{TV}_n \succeq \ln n / n$.
\end{exa}

\begin{exa}\label{excf}
Let $G$ be a group of exponential growth with controlled F\o lner pairs (see Tessera \cite[Definition 4.8]{Tes-iso}).
The results of Coulhon, Grigor'yan \& Pittet \cite{CGP}, which also apply to graphs, can be used to show that $P^{2n}(\mf{o}) \geq \mr{exp}(-Cn^{1/3})$, the slowest decay allowed by exponential growth.
Hence Corollary \ref{cretgrowth} gives $\mr{TV}_n \succeq n^{-1/2}$ (and so does Corollary \ref{tlbgradc}, see Remark \ref{rcretvsiis}), while the bound $H(P^n) \preceq \sqrt n$ of Peres \& Zheng \cite{PZ} and the entropy-increment inequality give $\mr{TV}_n \preceq n^{-1/4}$. On $F \wr \zz$, the argument of Example \ref{exz2z2} shows that $H(P^n)$ is comparable to the expected range of the lamplighter, so that $H(P^n) \asymp \sqrt n$ and Proposition \ref{prevpinsker} also gives $\mr{TV}_n \succeq n^{-1/2}$.
For $F \wr \zz$ this is the right order, at least for the switch-walk-switch measure (switch the lamp at the position of the lamplighter to a uniform element of $F$, make a step of the lazy simple random walk on $\zz$, and switch again). Indeed, conditionally on the trajectory $(x_k)$ of the lamplighter, the lamps on its range are then independent and uniform, the others being trivial. So left multiplication by a lamp generator, which only modifies the lamp at the origin, leaves $P^n$ unchanged, while $\|\delta_t * P^n - P^n\|_{\mr{TV}}$, for the position generator $t$, is at most the total variation distance between the law of $\big(\min_{k \leq n} x_k, \max_{k \leq n} x_k, x_n\big)$ and its translate,
which is $O(n^{-1/2})$ by the reflection principle and the local limit theorem. With Lemma \ref{lquot}, $\mr{TV}_n \asymp n^{-1/2}$: the total variation profile of $F \wr \zz$ is that of its projection to $\zz$.
\end{exa}

\begin{exa}\label{exzz}
Consider the lazy simple random walk on $\zz \wr \zz$ with respect to the standard generators (add $\pm 1$ to the lamp at the position of the lamplighter, or move the lamplighter), and write $X_n = (\varphi_n, x_n)$. Then $H(P^n) \asymp \sqrt n \ln n$ (see Erschler \cite{Ers-drift}), so that the increments of the entropy are of order $\ln n/\sqrt n$ by concavity, as in Example \ref{exz2z2}; the speed is $L(n) \asymp n^{3/4}$ (Erschler \cite{Dyu}, Revelle \cite{Revelle}); and the F\o lner function is $\asymp k^k$ (Erschler \cite{Ers-isop}), \ie $\Gd(x) \asymp \ln\ln x/\ln x$. Hence Proposition \ref{ptvspeed} gives $\mr{TV}_n \succeq n^{-3/4}$, Proposition \ref{pentprofile}(a) and the quotient $\zz$ (Lemma \ref{lquot}) give $\mr{TV}_n \succeq n^{-1/2}$, and Proposition \ref{prevpinsker} gives $\mr{TV}_n \succeq \ln n/\sqrt n$.
The truth is larger. Left multiplication by a lamp generator adds $1$ to the lamp at the origin, which is a simple random walk run for a number of steps of order $\sqrt n$; as the total variation of a sequence $k \mapsto \mu(k)$ is at least $2\max_k \mu(k)$, projecting on this lamp gives $\mr{TV}_n \geq \mb{P}\big(\varphi_n(0) = 0\big) \asymp n^{-1/4}$. Together with the entropy-increment inequality, $n^{-1/4} \preceq \mr{TV}_n \preceq n^{-1/4}\sqrt{\ln n}$: although the growth is exponential, the upper end of $H(P^{n+1}) - H(P^n) \preceq \mr{TV}_n \preceq \big(H(P^{n+1}) - H(P^n)\big)^{1/2}$ is attained, the discrepancy being a small displacement of a lamp value rather than a rare event.
\end{exa}

It thus seems that the lower bounds are met when the optimal F\o lner sets are still relatively close to balls, whereas the upper bound (entropy increment) is rather useful when the optimal F\o lner sequences are not ball-like.

\begin{rmk}\label{rfloor}
There are currently no realisation theorems (like the construction of prescribed speed, entropy and isoperimetric profiles of Brieussel \& Zheng \cite{BZ}) for $\mr{TV}_n$, so which functions occur is open.

The bound $n^{-1/2}$ looks like a floor. On groups of polynomial growth there are many ways to see this (\eg Corollary \ref{tlbgradpol-c}), and by Proposition \ref{ptvspeed} it holds whenever the walk is diffusive.
Corollaries \ref{cretgrowth} and \ref{tlbgradc} give it when the volume growth is at least $\mr{exp}(cn^\nu)$ and the return probability decays as slowly as this growth allows, \ie $P^{2n}(\mf{o}) \geq \mr{exp}(-Cn^{\nu/(\nu+2)})$.
This covers the groups of Example \ref{excf}.
The floor also holds on every group with an infinite virtually nilpotent quotient, by Lemma \ref{lquot} and Proposition \ref{ptvspeed} applied on the quotient, where the walk is diffusive (see Hebisch \& Saloff-Coste \cite{HSC}).

Our lower bounds restrict the groups on which this floor could fail. Assume, to fix ideas, that $L(n) \asymp n^\beta$ and $H(P^n) \asymp n^\eta$, and let $\nu$ be such that the group has $\IIS_\nu$. Proposition \ref{ptvspeed} gives the floor unless $\beta > 1/2$. Proposition \ref{prevpinsker} gives it unless $\eta < 1/2$, the increments of the entropy being then of order $n^{\eta - 1}$ by concavity (as in Example \ref{exz2z2}). Corollary \ref{tlbgradc} gives it unless $\eta > \nu/2$. Hence a group on which the floor fails satisfies $\nu/2 < \eta < 1/2 < \beta$. In particular $\nu < 1$, so the group does not have $\IIS_1$; as groups of exponential growth have $\IIS_1$ and groups of polynomial growth satisfy the floor, it has intermediate growth. Moreover, as recalled above, $L(n)^2 \preceq n\, H(P^n)$ by the Varopoulos--Carne estimate, whence $\beta \leq (1+\eta)/2 < 3/4$, and by the previous paragraph the group has no infinite virtually nilpotent quotient. Put differently, on every group with $\IIS_1$ (for instance of exponential growth) whose entropy grows like a power, the floor holds: Corollary \ref{tlbgradc} (with $\nu = 1$) gives it when $\eta \leq 1/2$, and Proposition \ref{prevpinsker} when $\eta \geq 1/2$.

This applies in particular to the first Grigorchuk group. Its simple random walk is superdiffusive, $n^{1/(2\alpha_0)} \preceq L(n) \preceq n^{3/4}$ where $\alpha_0 \approx 0.767$ is the exponent of its growth, so that Proposition \ref{ptvspeed} does not give the floor; but its F\o lner function is at least exponential (see Erschler \& Zheng \cite[Example 1.3]{EZ-iso}), \ie it has $\IIS_1$ (see Zheng \cite[Propositions 1.4 and 1.5]{Zheng-Grig}). Likewise, the finitely generated simple amenable groups of Juschenko \& Monod \cite{JM}, namely the derived subgroups of topological full groups of minimal subshifts, have exponential growth by Matui \cite[Corollary 2.5]{Matui}. The question thus remains open only for groups which, together with all their infinite quotients, have intermediate growth without $\IIS_1$, a superdiffusive speed (though slower than $n^{3/4}$) and an entropy growing more slowly than $\sqrt n$, \ie with a significant gap between entropy and speed. The author knows of no group on which $\mr{TV}_n \succeq n^{-1/2}$ fails.
\end{rmk}

\begin{ques}
Is there a group where $\mr{TV}_n \not\succeq n^{-1/2}$?
\end{ques}

We finally come back to the R\'enyi entropies of \S{}\ref{ssrenyi}. The functions $H_0(P^n)$, $H_1(P^n)$, $H_2(P^n)$ and $H_\infty(P^{2n})$ are all subadditive. Furthermore $H_1(P^n)$, $H_2(P^n)$ and $H_\infty(P^{2n})$ are concave (see below).
\begin{ques}\label{qrent}
Assume $G$ is a group of exponential growth. Is the function $n \mapsto H_0(P^{2n})$ concave? In other words, is the function $n \mapsto \ln |B_{2n}|$ (where $B_n$ are the balls of radius $n$) concave?
\end{ques}
It seems likely that the answer is no (the function $n \mapsto \ln |B_{n}|$ is not always concave, see Machi \cite{Machi}).
However, a positive answer would mean that numerical estimates on the [exponential] growth factor of groups are much more reliable.
(This is because the ratio test converges much more quickly than the root test, and concavity is what is needed for the ratio test to be a monotone sequence.)

The concavity of $H_1(P^n)$ comes from conditional entropy: by the chain rule and the invariance of the walk, $H_1(P^{n+1}) - H_1(P^n) = I(X_1 ; X_{n+1})$, which is non-increasing in $n$ by the data-processing inequality (see \S{}\ref{sstv}).
Conditional R\'enyi entropies have been defined, the most studied being that of Arimoto \cite{Arimoto}, and competing definitions are compared by Fehr \& Berens \cite{FB}. Arimoto's version is monotone under conditioning and satisfies a weak chain rule (the entropy lost by conditioning on $Y$ is at most $\ln |Y|$), but none of the proposed definitions satisfies the full chain rule \cite{FB}. Since it is precisely the chain rule, through the identity above, which the argument for $H_1$ uses, that argument does not transfer.
For $p=2$ and $p=\infty$ concavity holds for a different, spectral reason: $H_2(P^n) = H_\infty(P^{2n}) = -\ln P^{2n}(\mf{o})$, and $t \mapsto \ln P^{2t}(\mf{o})$ is convex (see the proof of Corollary \ref{cretgrowth}). We do not know whether $H_p(P^n)$ is concave for other values of $p \in \, ]0,1[$.

In general, these functions are only considered up to linear factors (\ie $f(x)$ and $Af(ax+b)+B$, where $A,B,a$ and $b \in \rr_{>0}$, are considered equivalent). Otherwise, they are not an invariant of the group, but depend on the choice of $P$. Note that it is still unknown whether or not [the equivalence class of] $H_1(P^n)$ depends on the choice of $P$.

Lemma \ref{lrenyi} shows that the functions $H_p(P^n)$ give, for $1 < p \leq \infty$, the same invariant.
\begin{ques}
How many different equivalence classes of functions are there in the sequence $H_p(P^n)$ where $-\infty \leq p < 1$?
\end{ques}
Note that, for $-\infty \leq p < 0$, one has $H_p(P^n) \asymp n$ on every infinite group, so that all these functions belong to the same equivalence class: for the lazy walk the smallest non-zero value of $P^n$ lies between $(\min P)^n$ and $2^{-n}$, while $\ln |\mr{supp}\, P^n| \preceq n$. Hence the interesting range is probably $0< p < 1$. If there happened to be many equivalence classes inside that range, there might be many interesting related questions: what are the critical values of $p$ where the equivalence class changes? does this value depend on $P$?

\begin{rmk}\label{rrenyismallp}
For $0 < p < 1$ the functions $H_p(P^n)$ are probably equivalent to $H_0(P^n) = \ln |B_n|$ on groups of polynomial or of exponential growth, but not on groups of intermediate growth. Recall that R\'enyi entropies are non-increasing in $p$, so that $H_1 \leq H_p \leq H_0$ for $0<p<1$.
\begin{itemize}
\item[$\bullet$] For a group of polynomial growth of degree $d$, the two-sided Gaussian estimates (Hebisch \& Saloff-Coste \cite{HSC}) give $H_p(P^n) = \tfrac d2 \ln n + O(1)$ for every $p > 0$, whereas $H_0(P^n) \asymp d \ln n$; these are equivalent.
\item[$\bullet$] For a group of exponential growth $H_0(P^n) \asymp n$, and the sum $\sum_x P^n(x)^p$ is dominated by large deviations: the points at distance $\rho n$ carry a total probability of order $e^{-c\rho^2 n}$, but they are exponentially numerous and $x \mapsto x^p$ favours small values. On the lamplighter $(\zz/2\zz) \wr \zz$ this heuristic can be checked: forcing the lamplighter to travel a distance $\rho n$ costs a probability of order $e^{-c\rho^2 n}$ and leaves the lamps essentially uniform on about $\rho n$ sites, which suggests $\sum_x P^n(x)^p \geq \mr{exp}\big(n[(1-p)\rho \ln 2 - pc\rho^2]\big)$, and choosing $\rho$ small gives $H_p(P^n) \succeq n$ for every $p < 1$.
\item[$\bullet$] If $|B_r| \leq \mr{exp}(Cr^V)$ with $V < 1$, the estimate of Varopoulos \& Carne, $P^n(x) \leq 2e^{-|x|^2/2n}$ \cite{Var,Carne}, and H\"older's inequality on spheres give $H_p(P^n) \leq C_p\, n^{V/(2-V)}$ for $0 < p < 1$. This is $o(n^V)$, so if moreover $|B_r| \geq \mr{exp}(cr^V)$, then $H_p(P^n)$ is \emph{not} equivalent to $H_0(P^n) \asymp n^V$. The same heuristic suggests that $H_p(P^n) \asymp n^{V/(2-V)}$ for every $p \in \, ]0,1[$ in that case.
\end{itemize}
In all three cases the range $0<p<1$ would then consist of a single equivalence class, determined by the volume growth and by the Gaussian decay of $P^n$ rather than by the typical behaviour of the walk; it coincides with the class of $H_0$ exactly when the growth is polynomial or exponential.
\end{rmk}

\end{document}